\documentclass[11pt]{amsart}

\usepackage[T1]{fontenc}
\usepackage{lmodern}
\usepackage[a4paper,margin=1in]{geometry}
\usepackage{amsmath,amssymb,mathtools}
\usepackage{booktabs}
\usepackage{array}
\usepackage{enumitem}
\usepackage{microtype}
\usepackage{xcolor}
\usepackage{aliascnt}
\usepackage{hyperref}
\usepackage[nameinlink,noabbrev]{cleveref}
\usepackage{tikz-cd}
\usetikzlibrary{arrows.meta,calc,matrix}

\hypersetup{
  colorlinks=true,
  linkcolor=blue!55!black,
  citecolor=green!40!black,
  urlcolor=blue!65!black,
  pdftitle={Remarks on the Complex Structure on S2 times S4},
  pdfauthor={Ruiming Liang, Chenhan Liu, Yang Zhang}
}

\setlist[itemize]{leftmargin=2em}
\setlist[enumerate]{leftmargin=2.3em}
\allowdisplaybreaks
\newtheorem{theorem}{Theorem}[section]
\newaliascnt{proposition}{theorem}
\newtheorem{proposition}[proposition]{Proposition}
\aliascntresetthe{proposition}
\newaliascnt{lemma}{theorem}
\newtheorem{lemma}[lemma]{Lemma}
\aliascntresetthe{lemma}
\newaliascnt{corollary}{theorem}
\newtheorem{corollary}[corollary]{Corollary}
\aliascntresetthe{corollary}
\newaliascnt{definition}{theorem}
\newtheorem{definition}[definition]{Definition}
\aliascntresetthe{definition}
\newaliascnt{remark}{theorem}
\newtheorem{remark}[remark]{Remark}
\aliascntresetthe{remark}

\newcommand{\Z}{\mathbb Z}

\newcommand{\R}{\mathbb R}

\newcommand{\Pp}{\mathbb P}

\newcommand{\Pic}{\operatorname{Pic}}
\newcommand{\Bl}{\operatorname{Bl}}

\title{\textbf{Remarks on the Complex Structures on $\mathbb P^3$ and $S^2\times S^4$}}

\author[R. Liang]{Ruiming Liang}
\address[R. Liang]{Peking University}
\email{hiltonliang@stu.pku.edu.cn}
\author[C. Liu]{Chenhan Liu}
\address[C. Liu]{Tsinghua University}
\email{liu-ch22@mails.tsinghua.edu.cn}
\author[Y. Zhang]{Yang Zhang}
\address[Y. Zhang]{Rutgers University}
\email{yang.zhang18@rutgers.edu}

\begin{document}

\begin{abstract}
Assuming the validity of the recently proposed \textit{``A compact complex threefold fibred by tori over the projective line, and the
six-sphere''}, we construct an exotic complex structure on $\mathbb P^{3}$, distinct from the point-blowup structures of Huckleberry, Kebekus, and Peternell. Then we perform an Atiyah flop to produce a complex structure on the standard smooth manifold $S^{2}\times S^{4}$.
\end{abstract}



\setcounter{tocdepth}{1}
\maketitle
\tableofcontents

\section{Introduction}

After the case of $S^{4}$ had been treated by Ehresmann and Hopf \cite{Ehresmann,Hopf}, Borel and Serre, in 1953, proved that a positive-dimensional sphere $S^{n}$ admits an almost complex structure if and only if $n=2$ or $n=6$
\cite[\S15, Proposition~15.1]{BorelSerre}. Subsequently, in 1965, Sutherland \cite[Theorem~3.1]{Sutherland}  showed that a product of even-dimensional spheres is almost complex exactly when it is a product of copies of\[ S^2,\qquad S^6, \qquad S^2\times S^4.\]

Recently, a compact complex threefold $X_{\mathrm{sph}}$ was constructed in \cite[Construction~6.1 and Theorem~6.2]{S6Paper} and it is proved there that
\[X_{\mathrm{sph}}\cong_{\mathrm{diff}} S^{6}\]
as an oriented smooth manifold \cite[Theorem~8.1]{S6Paper}. The aim of this paper is to prove that $S^2\times S^4$ has a complex structure, see \cref{thm:Xplus-main}. Consequently, subject to the result of \cite{S6Paper}, every product of even-dimensional spheres which admits an almost complex structure also admits a complex structure.

Now we state our construction. \cite{S6Paper} also proves that $\mathrm{Aut}^0 (X_{\mathrm{sph}})\cong \mathbb C^{*}$
\cite[Proposition~9.23]{S6Paper}, and the fixed locus of $\mathrm{Aut}^0 (X_{\mathrm{sph}})$ is a rational curve \[C\cong\mathbb P^{1},\] where the two complex normal weights are $+1$ and $-1$
\cite[Proposition~9.24]{S6Paper}; every nontrivial finite subgroup has
the same fixed locus \cite[Remark~9.25]{S6Paper}.

It is useful to recall heuristically how this fixed curve sits in the singular fibre $W$.  The normalization of $W$ is the degree-six del Pezzo surface $dP_6$, whose toric boundary is a hexagon of $(-1)$-curves. $W$ is obtained by identifying each side with its opposite side, which produces three rational double curves 
\[C,\qquad D_-,\qquad D_-'.\] These three curves pass through two triple points $ Q$ and $ Q'$ \cite[Theorem~4.5 and Proposition~4.6]{S6Paper}. The $\mathrm{Aut}^0 (X_{\mathrm{sph}})$-action restricted on $W$ is the one-parameter subgroup associated with one edge direction. The double curve in that direction is fixed pointwise; the other two double curves are preserved and the $\mathbb C^{*}$-action on each curve is the rotation fixing $P$ and $Q$. On either of them the complement of $P$ and
$Q$ is a single free $\mathbb C^{*}$-orbit
\cite[Proposition~9.24 and Remark~9.25]{S6Paper}.  Thus the local picture at either triple point may be visualized as
\[
 W=\{xyz=0\},
 \qquad
 \lambda\cdot(x,y,z)=(\lambda x,\lambda^{-1}y,z).
\]
The $z$-axis is the fixed double curve, while the other two coordinate axes are invariant curves along which the action moves points toward one triple point or the other.

As proved in Proposition \ref{prop: Q3exotic}, \[\widetilde X:=\operatorname{Bl}_{C}(X_{\mathrm{sph}})
\]is a complex manifold diffeomorphic to $Q_3:=\left\{[z_0:\cdots:z_4]\in\mathbb P^4\;\middle|\;z_0^2+\cdots+z_4^2=0\right\}$. We will show in Proposition~\ref{prop: Q3exotic} that
$\widetilde X$ is not biholomorphic to $Q_3$.

Let $\mu_2$ be the (unique) subgroup of $\mathrm{Aut}^0 (X_{\mathrm{sph}})\cong \mathbb C^{*}$ of order $2$, and $\iota\in\mu _2$ be the nontrivial element.
The action of $\iota$ on both normal directions to $C$ is multiplication
by $-1$. 
Blowing up the fixed curve
and taking the quotient gives a smooth complex manifold
\[
 X_{-}:=\operatorname{Bl}_{C}(X_{\mathrm{sph}})/\mu _2.
\]

We shall prove that $X_{-}$ is diffeomorphic to $\mathbb P^{3}$. On the other hand, the construction of Huckleberry, Kebekus, and Peternell \cite{HKP}, when applied to $X_{\mathrm{sph}}$, produces an exotic complex structure on $\mathbb P^{3}$ by blowing up a point of $X_{\mathrm{sph}}$. The following theorem distinguishes our complex structure from theirs.

\begin{theorem}\label{thm:Xminus-main}
The complex manifold $X_{-}$ is orientation-preservingly
diffeomorphic to $\mathbb P^{3}$. Moreover, it is neither biholomorphic nor deformation equivalent to the standard complex structure on
$\mathbb P^{3}$ and the point-blowup complex structures on
$\mathbb P^{3}$ obtained from $X_{\mathrm{sph}}$.
\end{theorem}

Let $C_-$ be the image of $(\mathrm{Bl}_C)_*^{-1}(D_-)$ in $X_-$. Then $C_-$ is a smooth rational curve with normal bundle $\mathcal O(-1)\oplus\mathcal O(-1)$ (\cref{lem: flop-curve}). Thus, we can apply an Atiyah flop to $C_-$ to obtain a complex manifold
\[
 X_{+}:=\operatorname{Flop}_{C_{-}}(X_{-}).
\]

\begin{theorem}\label{thm:Xplus-main}
The complex manifold $X_{+}$ is diffeomorphic to $S^{2}\times S^{4}$.
\end{theorem}

\begin{corollary}
    If $X\cong_{\mathrm{diff}} \prod_i S^{2n_i}$ is a product of even-dimensional spheres, then $X$ admits an almost complex structure if and only if $X$ admits a complex structure.
\end{corollary}

\begin{remark}

Our construction can be summarized by the following diagram:
\[
\begin{tikzcd}[
    row sep=0.8cm,
    column sep=0.75cm
]
&
\begin{array}{c}
\llap{$Q_3\cong_{\mathrm{diff}}\;$}\widetilde X
\end{array}
\arrow[dl, "\operatorname{Bl}_{C}"']
\arrow[dr, "/\mu_2"]
&&
\widehat X
\arrow[dl, "\operatorname{Bl}_{C_-}"']
\arrow[dr, "\operatorname{Bl}_{C_+}"]
&
\\
\begin{array}{c}
\llap{$S^6\cong_{\mathrm{diff}}\;$}X_{\mathrm{sph}}
\end{array}
&&
\begin{array}{c}
\llap{$\mathbb P^3\cong_{\mathrm{diff}}\;$}X_-
\end{array}
\arrow[rr, dashed, "\mathrm{Atiyah\ flop}"]
&&
\begin{array}{c}
X_+\rlap{$\;\cong_{\mathrm{diff}}S^2\times S^4$}
\end{array}
\end{tikzcd}
\]

\end{remark}

\subsection*{The strategy of proof}

By Wall's classification \cite{Wall}, a compact orientable simply connected spin six-dimensional differentiable manifold $Y$ with torsion-free cohomology groups can be distinguished by the cubic form $H^2(Y)\to H^6(Y):x\mapsto x\cup x\cup x$ and the first Pontryagin functional $p_{Y}(x):=\langle p_{1}(Y)\cup x,[Y]\rangle $. Hence, we only need to compute these topological invariants of $X_-$ and $X_+$ and compare them to those of $\mathbb P^3$ and $S^2\times S^4$ respectively.

To show that the complex structure obtained is an exotic one, we compute the first Chern class of $\widetilde X$ (or $X_-$) and compare them to the first Chern class of the standard complex structures on $Q_3$ (or $\mathbb P^3$ and exotic $\mathbb P^3$ obtained by blowing up one point on $X_\mathrm{sph}$ respectively).\\

The paper is organized as follows.
\begin{itemize}
 \item In Section~\ref{sec:Xminus-topology}, we construct $X_{-}$ as
    the $\mu_{2}$-quotient of the blow-up of $X_{\mathrm{sph}}$ along
    the fixed curve $C$, and determine its fundamental group and integral
    homology.

    \item In Section~\ref{sec:Xplus-topology}, we identify the distinguished
    curve $C_{-}\subset X_{-}$, define $X_{+}$ by performing the Atiyah
    flop along $C_{-}$, and transport the topological calculations across
    the common blow-up.

    \item In Section~\ref{sec:wall-invariants}, we compute the cubic form,
    the second Stiefel--Whitney class, and the first Pontryagin functional
    for both $X_{-}$ and $X_{+}$.

    \item In Section~\ref{Proof of the main theorems}, we apply
    Wall's classification theorem on certain six-manifolds.  We first identify the underlying
    oriented smooth manifold of $X_{-}$ with $\mathbb P^{3}$, and show
    that its complex structure is biholomorphically distinct from both the
    standard structure and the structure arising from the point-blow-up
    construction.  We then identify $X_{+}$ with
    $S^{2}\times S^{4}$, thereby obtaining the asserted complex structure
    on $S^{2}\times S^{4}$.

    \item In Section~\ref{Further discussion}, we give some further discussions.

\end{itemize}

\subsection*{Acknowledgment}
The first author would like to thank Professor Gang Tian for discussions. The third author is indebted to Professor Chi Li for his encouragement and help suggestions.

\subsection*{Statement on the use of AI}
At first, GPT 5.6 Sol used the fibration method similar to \cite{S6Paper} to construct a complex manifold $X$ fibreed over $\mathbb{P}^1$ and claimed that it is homeomorphic to $S^2\times S^4$. However, the proof provided by GPT 5.6 Sol is cumbersome, incomplete, and misses an essential lemma. Then by comparing the singular fibre of $X$ with the singular fibres of $X_{\mathrm{sph}}$, we discovered that that $X$ is likely to be obtained by the blowup-quotient-flop process introduced in this paper, which led to the main \cref{thm:Xplus-main} of this article. Wall’s classification result is also introduced to the authors by GPT 5.6 Sol.

The proofs of \cref{lem:normal-trivial}, \cref{lem: flop-curve}, \cref{prop:Xplus-p1} are first generated by GPT 5.6 Sol, but it is checked and rewrote by the authors.

This article is also polished by GPT 5.6 Sol on grammars, language, notations, conventions, and some graphs.

\subsection*{Notations and conventions}
Unless otherwise stated, all manifolds are smooth, connected, and oriented, and all homology and cohomology groups have integral coefficients. For a complex manifold $Y$, we write $c_i(Y):=c_i(TY)$, and similarly for $p_i(Y)$ and $w_i(Y)$. If $Z\subseteq Y$ is a smooth complex submanifold, then $\mathcal N_{Z/Y}$ denotes its holomorphic normal bundle and $\operatorname{Bl}_Z(Y)$ denotes the blow-up of $Y$ along $Z$.

If $D$ is a Cartier divisor on $Y$, then $[D]:=c_1\bigl(\mathcal O_Y(D)\bigr)\in H^2(Y )$,
whereas $[Y]$ denotes the fundamental class of an oriented closed manifold. Numerical expressions such as $D^3$ mean
$\langle[D]^3,[Y]\rangle$. The notation $\cong_{\mathrm{diff}}$ denotes
an orientation-preserving diffeomorphism, and $\mathbb P^n$ denotes complex projective $n$-space.

The global geometric notation is summarized in the following table; symbols used only within a single calculation are defined at their first occurrence.
\begin{center}
\small
\renewcommand{\arraystretch}{1.15}
\begin{tabular}{@{}>{$}l<{$}@{\qquad}p{0.66\textwidth}@{}}
\toprule
\text{Notation} & \text{Meaning}\\
\midrule
X_{\mathrm{sph}} & the compact complex threefold serving as the source of the construction;\\
W,\ E_0,\ \nu & the cusp fibre, its normalization $E_0\simeq dP_6$, and the normalization map $\nu\colon E_0\to W$;\\
C,\ D_-,\ D_-',\  Q,\  Q' & the fixed double curve, the other two double curves of $W$, and the two triple points;\\
\mu_2=\{1,\iota\} & the order-two subgroup of $\operatorname{Aut}^0(X_{\mathrm{sph}})\cong\mathbb C^*$;\\
\widetilde X,\ E & $\widetilde X:=\operatorname{Bl}_C(X_{\mathrm{sph}})$ and its exceptional divisor;\\
q,\ X_-,\ F & the quotient map $q\colon\widetilde X\to X_-$ and its branch divisor $F=q(E)$;\\
\widetilde D_-,\ C_- & the strict transform of $D_-$ in $\widetilde X$ and its image $C_-=q(\widetilde D_-)$;\\
\widehat X,\ R & $\widehat X:=\operatorname{Bl}_{C_-}(X_-)$ and its exceptional divisor;\\
X_+,\ C_+ & the Atiyah flop of $X_-$ and the image of $R$ in $X_+$;\\
\mathcal L,\ \ell & the branch half-line bundle on $X_-$ and its class $\ell=c_1(\mathcal L)$;\\
\mathcal L_+,\ \ell_+ & the corresponding line bundle on $X_+$ and its class $\ell_+=c_1(\mathcal L_+)$;\\
Q_3 & the smooth quadric threefold in $\mathbb P^4$.\\
\bottomrule
\end{tabular}
\end{center}

\section{The topology of \texorpdfstring{$X_{-}$}{Xminus}}\label{sec:Xminus-topology}

\subsection{The construction of $X_-$}
\label{sec:source-resolution}

Near a point of $C$ there are holomorphic coordinates $(z,x,y)$, with
$z$ tangent to $C$, in which $\lambda\in \mathrm{Aut}^0 (X_{\mathrm{sph}})\cong \mathbb C^{*}$ acts as
\[
 \lambda\cdot(z,x,y)=(z,\lambda x,\lambda^{-1}y).
\]
Hence, the element
$\iota=-1\in\mathbb C^{*}$ acts as $\iota\cdot(z,x,y)=(z,-x,-y)$. Let
\[
 \mathrm{Bl}_C\colon\widetilde X=\operatorname{Bl}_{C}(X_{\mathrm{sph}})
 \longrightarrow X_{\mathrm{sph}}
\]
be the blow-up, with exceptional divisor $E$. In the blow-up
chart $y=xt$, the lifted involution $\widetilde \iota$ is
\[
 \widetilde \iota\cdot (z,x,t)=(z,-x,t).
\]
Its quotient has smooth coordinates $(z,s=x^{2},t)$. Hence, the action of $\widetilde \iota$ fixes $E$ pointwise and the quotient map
\[
 q\colon\widetilde X\longrightarrow
 X_{-}
\]
is a holomorphic double cover branched along the smooth divisor
\[
 F:=q(E).
\]

\subsection{The homology group of $X_-$}

We now compute the homology of $X_-$. Let \[M:=X_{\mathrm{sph}}\setminus C,\qquad A:= M/\mu_2\cong X_-\setminus F. \]
The homology groups of $M$ can be obtained by Alexander duality. Since $\mu_2$ acts freely on $M$, we can compute the homology groups of $A$ by the Cartan-Leray spectral sequence. Let $B$ be an open tubular neighborhood of the branch divisor $F\cong \mathbb P^1\times \mathbb P^1$ in $X_-$.
The homology of $B$ is clear, as $B$ is homotopy equivalent to $S^2\times S^2$. Then we can show $A\cap B\simeq S^2\times \mathbb{R}P^3$, $ A\cup B=X_-$ and use the MV sequence to compute the homology of $X_-$.

\begin{lemma}\label{lem:exterior}
The manifold $M=X_{\mathrm{sph}}\setminus C$ is simply connected and
\[
H_{k}(M)=
\begin{cases}
\Z,&k=0,3,\\
0,&\text{otherwise}.
\end{cases}
\]
Moreover, $\iota$ acts trivially on $H_{3}(M)$.
\end{lemma}

\begin{proof}
Alexander duality gives
\[
       \widetilde H_{k}(M)
       \cong\widetilde H^{\,5-k}(S^{2}),
\]
which gives the desired homology. A loop in $M$ bounds a disc in
$S^{6}$. After a relative general-position perturbation, the disc is
disjoint from $C$, since $\dim D^{2}+\dim C=4<6$.
Thus $\pi_{1}(M)=0$. 

The generator of $H_{3}(M)$ is represented by a meridian $S^{3}$ around a point in $C$. The involution $\iota$ restricts to the
antipodal map on this $S^{3}$, whose degree is
$(-1)^{3+1}=+1$. Hence, the action of $\iota$ on $H_{3}(M)$ is trivial.
\end{proof}

\begin{proposition}\label{prop:A-homology}
The cohomology groups of $A$ are
\[
H^{k}(A)=
\begin{cases}
\Z,&k=0,3,\\
\Z/2\Z,&k=2,\\
0,&\text{otherwise},
\end{cases}
\]
and the homology groups of $A$ are
\[
H_{k}(A)=
\begin{cases}
\Z,&k=0,3,\\
\Z/2\Z,&k=1,\\
0,&\text{otherwise}.
\end{cases}
\]
\end{proposition}

\begin{proof}
For a free cellular action of a discrete group $G$ on a CW complex
$X$, the cohomological Cartan--Leray spectral sequence is
\[
 E_{2}^{p,q}=H^{p}\!\left(G;H^{q}(X )\right)
 \Longrightarrow H^{p+q}(X/G ).
\]
Here $H^{p}(G;-)$ denotes group cohomology. Apply this with $G=\mathbb Z/2\mathbb Z$ and $X=M$. We obtain
\[
E_{2}^{p,q}
=
H^{p}\bigl(\Z/2\Z;H^{q}(M)\bigr)
\Longrightarrow H^{p+q}(A).
\]
Only the rows $q=0,3$ occur, and both coefficient modules are trivial by \cref{lem:exterior}. The group cohomology $H^*(\mathbb Z/2\mathbb Z; \mathbb Z)$ with trivial module structure can be represented as
\[
H^{*}(\Z/2\Z;\Z)=\Z[u]/(2u),\qquad \deg u=2,
\]
i.e.,
\[
 H^{p}(\mathbb Z/2\mathbb Z; \mathbb Z)=
 \begin{cases}
 \mathbb Z,&p=0,\\
 \mathbb Z/2\mathbb Z,&p>0\text{ even},\\
 0,&p\text{ odd}.
 \end{cases}
\]
Let $x$ generate $H^{3}(M)$.  Symbolically, the two nonzero
rows are
\[
 E_{2}^{*,*}
 =
 \mathbb Z[u]/(2u)
 \oplus x\,\mathbb Z[u]/(2u),
\qquad
 \deg u=(2,0),\qquad \deg x=(0,3).
\]

The differential operators in the spectral sequence are $d_r\colon E_r^{p,q}\longrightarrow E_r^{p+r,q-r+1}$. 
The only possible nonzero differential is therefore $d_{4}\colon E_{4}^{p,3}\longrightarrow E_{4}^{p+4,0}$. 

\begin{equation*}
\begin{tikzpicture}[
  baseline=(current bounding box.center),
  >={Latex},
  every node/.style={font=\scriptsize}
]
\matrix (e4) [
  matrix of math nodes,
  nodes in empty cells,
  row sep=.9em,
  column sep=1.45em,
  row 1/.style={nodes={minimum height=3.2em}}
] {
 q=3 & \mathbb Z & 0 & \mathbb Z/2 & 0 & \mathbb Z/2 & \cdots \\
 q=0 & \mathbb Z & 0 & \mathbb Z/2 & 0 & \mathbb Z/2 & \cdots \\
     & p=0 & 1 & 2 & 3 & 4 & \\
};
\coordinate (vline) at ($(e4-1-1.east)!.5!(e4-1-2.west)$);
\draw[thin]
  (vline |- e4-1-1.north)
  -- (vline |- e4-3-1.south);
\draw[thin]
  ($(e4-2-1.south west)!.5!(e4-3-1.north west)$)
  -- ($(e4-2-7.south east)!.5!(e4-3-7.north east)$);
\draw[->,very thick,black!70!black]
  (e4-1-2) --
  node[midway,above,sloped,fill=white,inner sep=1pt] {$d_4$}
  (e4-2-6);
\end{tikzpicture}
\end{equation*}

We claim that $d_4$ at $p=0$ is a nonzero homomorphism
\[
       \mathbb Z\longrightarrow\mathbb Z/2\mathbb Z:d_{4}(x)=u^{2}.
\]
If $d_4(x)=0$, multiplicativity would give
\[
       d_{4}(u^{k}x)=u^{k}d_{4}(x)=0
\]
for every $k$. Hence, all the classes $u^{k}$ would survive in
arbitrarily large total degrees. This is impossible because $H^k(A)$ must be $0$ when $k>6$. Thus $d_4(x)=u^{2}$. Multiplicativity gives
\[
 d_4(u^{p}x)=u^{p+2}.
\]
For every $p\geq1$, this is an isomorphism
$E_{4}^{p,3}\to E_{4}^{p+4,0}:\mathbb Z/2\to\mathbb Z/2$.  It follows that $u$ survives in degree $2$, every $u^{p}$ with $p\geq 2$ is killed, and every positive-column class in the upper row is killed.  The surviving subgroup
at $(p,q)=(0,3)$ is
\[
       \ker\bigl(\Z\xrightarrow{\bmod 2}\Z/2\bigr)=2\Z\cong\Z.
\]
Thus the only nonzero $E_{\infty}$-terms are
\[
 E_{\infty}^{0,0}=\mathbb Z,\qquad
 E_{\infty}^{2,0}=\mathbb Z/2\mathbb Z,\qquad
 E_{\infty}^{0,3}=2\mathbb Z\cong\mathbb Z.
\]
This proves the cohomology statement.  Notice that the
surviving class in degree three restricts to twice a generator of
$H^{3}(M )$.

We finally recover homology by universal coefficient theorem. Since $M$ is a double cover of $A$, we have $\pi_{1}(A)\cong\Z/2\Z$, and so
\[
 H_{1}(A )=\mathbb Z/2\mathbb Z.
\]
The universal coefficient sequence in degree two is
\[
 0\longrightarrow
 \operatorname{Ext}(H_{1}(A),\mathbb Z)
 \longrightarrow H^{2}(A )
 \longrightarrow\operatorname{Hom}(H_{2}(A),\mathbb Z)
 \longrightarrow0.
\]
The first two groups are both $\mathbb Z/2\mathbb Z$, so
$\operatorname{Hom}(H_{2}(A),\mathbb Z)=0$.  The degree-three and
degree-four universal coefficient sequences show respectively that
$\operatorname{Ext}(H_{2}(A),\mathbb Z)=0$ and that $H_{3}(A)$ has
no torsion.  As all homology groups are finitely generated, this forces
\[
 H_{2}(A)=0,\qquad H_{3}(A)=\mathbb Z.
\]
Continuing in degrees four, five, and six gives
\[
 H_{4}(A)=H_{5}(A)=0
\]
and shows that $H_{6}(A)$ is torsion. Since $A$ is homotopy equivalent to a compact,
connected, oriented six-manifold with nonempty boundary,
$H_{6}(A )=0$.
This proves the homology table.
\end{proof}

Next, we compute the homology groups of $B$ and the intersection $A\cap B$.

By \cite[Proposition~9.24 and Remark~9.25]{S6Paper}, the fixed locus of $\iota$ is a smooth curve $C=X_{\mathrm{sph}}^{\iota}\cong \mathbb P^1$, 
and $\iota$ acts by $-I$ on the normal bundle
$\mathcal{N}_{C/X_{\mathrm{sph}}}$.

\begin{lemma}\label{lem:normal-trivial}
We have $\mathcal{N}_{C/X_{\mathrm{sph}}}\cong \mathcal{O}(-1)\oplus \mathcal{O}(-1)$. In particular, as an oriented real rank-four bundle, the normal bundle $\mathcal{N}_{C/X_{\mathrm{sph}}}\to C\cong_{\mathrm{diff}} S^{2}$ is trivial.
\end{lemma}

\begin{proof}
Following the notation in \cite[Proposition~9.24]{S6Paper}, let
\[
 \nu\colon E_{0}\simeq dP_{6}\longrightarrow W
\]
be the normalization of the cusp fibre. By
\cite[Lemma~4.2(iv) and Proposition~4.6(ii)--(iii)]{S6Paper}, the two
branches over $C$ are opposite boundary curves
$C_{v},C_{-v}\subseteq E_{0}$; each map isomorphically to $C$ and is a
$(-1)$-curve. Hence
\[
 \mathcal{N}_{C_{v}/E_{0}}\cong \mathcal{N}_{C_{-v}/E_{0}}
 \cong\mathcal O_{\mathbb P^{1}}(-1).
\]
The normal-crossings equations $xy=0$ along the smooth part of $C$
and $xyz=0$ at its two triple points
\cite[Theorem~4.5(d)]{S6Paper} show that these two branch-normal
directions extend across the triple points and give an exact sequence
\[
 0\longrightarrow\mathcal O(-1)
 \longrightarrow \mathcal{N}_{C/X_{\mathrm{sph}}}
 \longrightarrow\mathcal O(-1)\longrightarrow0.
\]
This sequence splits because
$\operatorname{Ext}^{1}(\mathcal O(-1),\mathcal O(-1))
=H^{1}(\mathbb P^{1},\mathcal O)=0$, proving the holomorphic assertion.
Finally, the underlying real bundle has
$w_{2}=c_{1}(\mathcal{N}_{C/X_{\mathrm{sph}}})\bmod2=(-2)\bmod2=0$, and an oriented
rank-four bundle over $S^{2}$ with vanishing $w_{2}$ is trivial.
\end{proof}

Hence, $F\cong E\cong \mathbb P^1\times \mathbb P^1$. So $\pi_{1}(B)=0$ and 
\[H_{k}(B)=
\begin{cases}
\Z,&k=0,4,\\
\Z^2,&k=2,\\
0,& \text{otherwise}.
\end{cases}
\]

Choose an open neighborhood $N(C)$ of $C$ in $X_{\mathrm{sph}}$ such that
\[N(C)\cong_{\mathrm{diff}} S^{2}\times D^{4},
\qquad
\iota(c,v)=(c,-v).
\]
Then $N(C)\setminus C$ can be identified with a neighborhood of $E$ in $\widetilde X$ minus $E$. Since $\partial N(C)\cong_{\mathrm{diff}} S^{2}\times S^{3}$ and $\iota$ acts on the $S^3$ factor by the antipodal map, we have
\[
     P:=\partial \mathcal{N}_{F/X_-}\cong_{\mathrm{diff}} \partial \mathcal{N}_{C/X_{\mathrm{sph}}}\big{/}\mu_2\cong_{\mathrm{diff}} S^2\times \mathbb RP^3.
\]
Note that $A\cap B\simeq P$, the Künneth theorem gives
\[
H_{k}(A\cap B)\cong H_{k}(P)=
\begin{cases}
\Z,&k=0,2,5,\\
\Z/2\Z,&k=1,\\
\Z\oplus\Z/2\Z,&k=3,\\
0,&k=4.
\end{cases}
\]

\begin{lemma}\label{lem:boundary-maps}
The inclusion $P\hookrightarrow A$ induces:
\begin{enumerate}
  \item an isomorphism on fundamental groups
  \[
       \pi_{1}(P)\xrightarrow{\ \sim\ }\pi_{1}(A),
  \]
  and hence an isomorphism
  \[
       H_{1}(P)\xrightarrow{\ \sim\ }H_{1}(A);
  \]
  \item an isomorphism from the free summand of $H_{3}(P)$ onto
  $H_{3}(A)$.
\end{enumerate}
\end{lemma}

\begin{proof}
A generator of $\pi_{1}(\R P^{3})$ lifts to a path in $S^{3}$ from
$v$ to $-v$, and therefore represents the nontrivial deck
transformation of $M\to A$. This proves the first assertion.

For the second, let
\[
       \alpha=[\{c\}\times\R P^{3}]
\]
generate the free summand of $H_{3}(P)$, let $a$ generate
$H_{3}(A)$, and let $g$ be the meridian generator of $H_{3}(M)$.
Write
\[
       i_{*}\alpha=ka,\qquad \operatorname{tr}(a)=mg.
\]
Transfer is natural for the square of double covers
\[
\begin{array}{ccc}
S^{2}\times S^{3}&\longrightarrow&M\\
\downarrow&&\downarrow\\
S^{2}\times\R P^{3}&\longrightarrow&A.
\end{array}
\]
The transfer of $\alpha$ is the $S^{3}$-fibre, which maps to
$\pm g$ by the meridian isomorphism. Hence
\[
       km=\pm1.
\]
Therefore $k=\pm1$, proving the assertion.
\end{proof}

\begin{lemma}\label{lem:H2-primitive}
The map
\[
       H_{2}(P)\longrightarrow H_{2}(B)
\]
is injective with torsion-free rank-one cokernel.
\end{lemma}

\begin{proof}
The neighborhood $B$ is the real rank-two disc bundle of the normal
complex line bundle of $F$, with sphere boundary $P$. By the Thom
isomorphism,
\[
       H_{3}(B,P)\cong H_{1}(F)=0,\qquad
       H_{2}(B,P)\cong H_{0}(F)=\Z.
\]
The long exact sequence of the pair contains
\[
0\longrightarrow H_{2}(P)
\longrightarrow H_{2}(B)
\longrightarrow\Z
\longrightarrow H_{1}(P)\cong\Z/2\Z
\longrightarrow0.
\]
The map $\Z\to H_{1}(P)\cong\Z/2\Z$ is surjective, so its kernel is
$2\Z$. It follows that
the first map is injective and its cokernel is isomorphic to
$2\Z\cong\Z$, hence is torsion-free.
\end{proof}

Now we apply the MV-sequence to $X_-=A\cup B$.

\begin{proposition}\label{prop:Xminus-homology}
The manifold $X_{-}=A\cup B$ is simply connected and
\[
H_{k}(X_{-})=
\begin{cases}
\Z,&k=0,2,4,6,\\
0,&k=1,3,5.
\end{cases}
\]
\end{proposition}

\begin{proof}
The van Kampen theorem gives
\[
\pi_{1}(X_{-})
\cong
\pi_{1}(A)*_{\pi_{1}(P)}\pi_{1}(B)=1,
\]
because $\pi_{1}(P)\to\pi_{1}(A)$ is an isomorphism and
$\pi_{1}(B)=1$.

Consider the MV-sequence
\begin{equation}\label{eq:MV-zigzag}
\begin{tikzpicture}[
  baseline=(current bounding box.center),
  >={Latex},
  descr/.style={fill=white,inner sep=1.5pt}
]
\matrix (m) [
  matrix of math nodes,
  row sep=1.4em,
  column sep=2.5em,
  nodes={text height=1.5ex,text depth=0.4ex,inner sep=1.5pt}
]
{
\cdots & H_{3}(A\cap B) & H_{3}(A)\oplus H_{3}(B) & H_{3}(X_{-}) \\
       & H_{2}(A\cap B) & H_{2}(A)\oplus H_{2}(B) & H_{2}(X_{-}) \\
       & H_{1}(A\cap B) & H_{1}(A)\oplus H_{1}(B) & \cdots \\
};
\path[overlay,->,font=\scriptsize]
  (m-1-1) edge (m-1-2)
  (m-1-2) edge (m-1-3)
  (m-1-3) edge (m-1-4)
  (m-1-4) edge[out=355,in=175,black!70!black]
    node[descr,yshift=.3ex] {$\partial$} (m-2-2)
  (m-2-2) edge (m-2-3)
  (m-2-3) edge (m-2-4)
  (m-2-4) edge[out=355,in=175,black!70!black]
    node[descr,yshift=.3ex] {$\partial$} (m-3-2)
  (m-3-2) edge (m-3-3)
  (m-3-3) edge (m-3-4);
\end{tikzpicture}
\end{equation}
Substituting the groups already calculated, this becomes
\[
\mathbb Z\oplus\mathbb Z/2
\xrightarrow{(\,\pm1,\,0\,)}
\mathbb Z
\longrightarrow H_{3}(X_{-})
\longrightarrow
\mathbb Z
\hookrightarrow
\mathbb Z^{2}
\longrightarrow H_{2}(X_{-})
\longrightarrow
\mathbb Z/2
\xrightarrow{\sim}
\mathbb Z/2.
\]
The first arrow is surjective by \cref{lem:boundary-maps}; its torsion
summand necessarily maps to zero.  Hence $H_{3}(X_{-})$ injects into
$H_{2}(P)=\mathbb Z$.  Its image is the kernel of the displayed
injection $\mathbb Z\hookrightarrow\mathbb Z^{2}$, and is therefore
zero.  Thus $H_{3}(X_{-})=0$.

At the other end, the last arrow is an isomorphism, so
$\mathbb Z^{2}\to H_{2}(X_{-})$ is surjective.  Its kernel is the
image of $H_{2}(P)$, and \cref{lem:H2-primitive} identifies the
cokernel of that image with $\mathbb Z$.  Therefore
\[
       H_{3}(X_{-})=0,\qquad H_{2}(X_{-})=\Z.
\]
Simple connectivity gives $H_{1}(X_{-})=0$.  Since $X_{-}$ is a
closed connected oriented six-manifold, $H_{0}(X_{-})=H_{6}(X_{-})=
\mathbb Z$.  The universal coefficient theorem gives
$H^{2}(X_{-})\cong\operatorname{Hom}(H_{2}(X_{-}),\mathbb Z)=
\mathbb Z$ and $H^{1}(X_{-})=0$.  Poincaré duality then gives
\[
       H_{4}(X_{-})\cong H^{2}(X_{-})\cong\Z,\qquad
       H_{5}(X_{-})\cong H^{1}(X_{-})=0,
\]
which completes the table.
\end{proof}

\section{The topology of \texorpdfstring{$X_{+}$}{X+}}\label{sec:Xplus-topology}

\subsection{The construction of
\texorpdfstring{$X_{+}$}{X+}}

For the same reason as \cref{lem:normal-trivial}, the normal bundle of the curves $D_-$ is isomorphic to $\mathcal{O}(-1)\oplus \mathcal{O}(-1)$. Let $\widetilde D_-\subseteq\widetilde X$ be its strict transform by the blow up $\mathrm{Bl}_C$. We claim
that
\begin{lemma}\label{lem: flop-curve}
    The image $C_{-}:=q(\widetilde D_{-})\subseteq X_{-}$ is a smooth rational curve and
    \[\mathcal{N}_{C_-/X_-}\cong\mathcal{O}(-1)\oplus \mathcal{O}(-1).\]
\end{lemma}

\begin{proof}
Identifying $\widetilde D_{-}$ with $D_{-}$ via the blow up map $b:\widetilde X=\operatorname{Bl}_{C}(X_{\mathrm{sph}})
\rightarrow X_{\mathrm{sph}}$ and blowing up the two
transverse intersection points $ Q', Q'$ gives an elementary exact sequence
\begin{equation*}
0\longrightarrow\mathcal N_{\widetilde D_{-}/\widetilde X}
\longrightarrow\mathcal N_{D_{-}/X_{\mathrm{sph}}}
\longrightarrow\mathbb C_{ Q}\oplus\mathbb C_{ Q'}
\longrightarrow0.
\end{equation*}
According to \cref{lem:normal-trivial}, $\mathcal N_{D_{-}/X_{\mathrm{sph}}}\cong \mathcal O(-1)\oplus \mathcal O(-1)$ .Thus the only possibilities of $\mathcal N_{\widetilde D_{-}/\widetilde X}$ are
\begin{equation}\label{eq:Dminus-two-splittings}
\mathcal O(-2)\oplus \mathcal O(-2)
\qquad\text{or}\qquad
\mathcal O(-1)\oplus\mathcal O(-3).
\end{equation}

We now exclude the second possibility. By \cite[Propositions~9.23--9.24 and Remark~9.25]{S6Paper}, 
\[
q|_{\widetilde D_{-}}\colon\widetilde D_{-}\longrightarrow C_{-}
\]
is a double cover branched at the points over $ Q, Q'$. The quotient is
smooth along $C_{-}$, and Riemann--Hurwitz gives
$C_{-}\cong\mathbb P^{1}$. Moreover, the normal direction coordinates near $\widetilde D_-$ descend unchanged, so $\mathcal N_{\widetilde D_{-}/\widetilde X}\cong\rho^{*}\mathcal N_{C_{-}/X_{-}}$. Write $\mathcal N_{C_{-}/X_{-}}
\cong\mathcal O(m)\oplus\mathcal O(n)$. Since $\deg q|_{\widetilde D_{-}}=2$, we have that
\[
\mathcal N_{\widetilde D_{-}/\widetilde X}
\cong\mathcal O(2m)\oplus\mathcal O(2n).
\]
Both splitting degrees upstairs are therefore even. This excludes
$\mathcal O(-1)\oplus\mathcal O(-3)$ in
\eqref{eq:Dminus-two-splittings}, and hence
\[
\mathcal N_{\widetilde D_{-}/\widetilde X}
\cong\mathcal O_{\mathbb P^{1}}(-2)\oplus \mathcal O_{\mathbb P^{1}}(-2).
\]
Consequently, $m=n=-1$, and the lemma follows.
\end{proof}

Therefore, we can apply the standard (analytic) Atiyah flop to $X_-$.

\begin{definition}\label{def:Xplus-flop}
Let
\[
\widehat X:=\operatorname{Bl}_{C_{-}}X_{-}\longrightarrow X_{-}
\]
be the blow-up, and let $R\subseteq \widehat X$ be its exceptional divisor.  Then
\[
 R\cong\mathbb P^{1}\times\mathbb P^{1},
 \qquad
 \mathcal O_{\widehat X}(R)|_R\cong\mathcal O_R(-1,-1).
\]
Contracting $R$ along another ruling gives a smooth compact complex threefold $X_{+}$ with a curve $C_+$ (the image of $R$) and a morphism
\[
 \mathrm{Bl}_{C_+}:\widehat X\longrightarrow X_{+}.
\]
We call $X_{+}$ the \emph{Atiyah flop} of $X_{-}$ along $C_{-}$.
\end{definition}

The contraction in the definition is the standard analytic Atiyah contraction \cite{Atiyah}. By construction, we have $X_-\setminus C_-=X_+\setminus C_+=\widehat X\setminus R$. \begin{equation}\label{eq:global-flop-diagram}
\begin{tikzcd}[column sep=4.8em,row sep=2.4em]
& \widehat{X}
  \arrow[dl,"{\operatorname{Bl}_{C_-}}"']
  \arrow[dr,"{\operatorname{Bl}_{C_+}}"] & \\
X_- && X_+
\end{tikzcd}
\end{equation}
Both arrows are blow-ups of a smooth rational curve with normal bundle
$\mathcal O(-1)\oplus\mathcal O(-1)$.

\subsection{The homology groups of $X_+$}

It is standard to compute the homology groups across the Atiyah flop. For completeness of the paper, we give a sketch here.

\begin{lemma}\label{lem:integral-blowup}
Let $a\colon\widetilde Y=\operatorname{Bl}_{C}Y\to Y$ be the blow-up
of a complex manifold along a smooth connected complex
codimension-two submanifold. Then
\[
 H^{k}(\widetilde Y )
 \cong H^{k}(Y )\oplus H^{k-2}(C ).
\]
\end{lemma}

\begin{proof}
Let $E=\mathbb P(\mathcal{N}_{C/Y})$, let $\pi\colon E\to C$ be the
projection, and let $j\colon E\hookrightarrow\widetilde Y$ be the
inclusion. Excision and the Thom isomorphism compare the long exact
sequences of $(Y,Y\setminus C)$ and
$(\widetilde Y,\widetilde Y\setminus E)$. Together with the integral
projective-bundle decomposition
\[
 H^{*}(E )
 =
 \pi^{*}H^{*}(C )
 \oplus
 \xi\,\pi^{*}H^{*-2}(C ),
 \qquad \xi=c_{1}(\mathcal O_E(1)),
\]
the comparison map is the identity off the exceptional locus and sends
the Thom summand of $C$ to the $\xi$-summand of $E$. Five Lemma applied to the two long exact sequences therefore gives what we want.
\end{proof}

\begin{lemma}
\label{lem:pi1-blowup}
Both the blow-up maps in the Atiyah flop induce isomorphisms
\[
 \pi_{1}(\widehat X)\cong\pi_{1}(X_{-}),
 \qquad
 \pi_{1}(\widehat X)\cong\pi_{1}(X_{+}).
\]
\end{lemma}

\begin{proof}
We treat $\widehat X \to X_{-}$. Let $T$
be a tubular neighbourhood of $C_{-}\cong\mathbb P^{1}$ and put
$U=X_{-}\setminus T$. The space $T$ retracts onto
$C_{-}$ and is simply connected. Its boundary is homeomorphic to $S^2\times S^3$, hence is simply connected.
The inverse image of $T$ retracts onto $R \cong\mathbb P^{1}\times\mathbb P^{1}$, so it is also simply connected. Its boundary is the same as
$\partial T$. Applying the van Kampen theorem proves the lemma.
\end{proof}

\begin{proposition}\label{thm:homology}
$X_{+}$ is simply connected and
\[
H_{k}(X_{+})=
\begin{cases}
\Z,&k=0,2,4,6,\\
0,&k=1,3,5.
\end{cases}
\]
\end{proposition}

\begin{proof}
Apply \cref{lem:integral-blowup} to both arrows of the flop:
\[
 H^{k}(\widehat X )
 \cong H^{k}(X_{-} )
       \oplus H^{k-2}(\mathbb P^{1} ),
\]
\[
 H^{k}(\widehat X )
 \cong H^{k}(X_{+} )
       \oplus H^{k-2}(\mathbb P^{1} ).
\]
The added group is $\mathbb Z$ for $k=2,4$, and is zero in every
other degree. By \cref{prop:Xminus-homology} and integral Poincaré duality,
the first decomposition gives
\[
 H^{k}(\widehat X )=
 \begin{cases}
 \mathbb Z,&k=0,6,\\
 \mathbb Z^{2},&k=2,4,\\
 0,&k=1,3,5.
 \end{cases}
\]
Comparing with the second decomposition degree by degree gives
\[
 H^{k}(X_{+} )=
 \begin{cases}
 \mathbb Z,&k=0,2,4,6,\\
 0,&k=1,3,5.
 \end{cases}
\]
For $k=2,4$, this cancels one free $\mathbb Z$-summand. The
classification of finitely generated abelian groups also shows that no
torsion can occur.

By \cref{lem:pi1-blowup,prop:Xminus-homology},
\[
 \pi_{1}(X_{+})\cong\pi_{1}(\widehat X)
 \cong\pi_{1}(X_{-})=1.
\]
The universal coefficient theorem, together with integral Poincaré
duality, converts the cohomology table into the stated homology table
and confirms that no hidden torsion occurs. 
\end{proof}

\begin{remark}\label{rem:not-diffeo-flop}
An flop preserve the cohomology and homology groups (as additive groups), but it may change the ring structure of the cohomology ring, thus the cubic form and the first Pontryagin functional. Those invariants are computed separately below.
\end{remark}

\section{The Wall invariants of \texorpdfstring{$X_{-}$ and $X_{+}$}{Xminus and X+}}\label{sec:wall-invariants}

\subsection{Characteristic classes of \texorpdfstring{$X_{-}$}{Xminus}}

Let
\[
 \mathrm{Bl}_C\colon \widetilde X:=\Bl_{C}(X_{\mathrm{sph}})\longrightarrow
 X_{\mathrm{sph}}
\]
be the blow-up of the fixed curve, and denote its exceptional divisor by
$E$.  The lifted involution $\widetilde \iota$ fixes $E$ pointwise, and gives a holomorphic double covering
\[
 q\colon\widetilde X\longrightarrow X_{-}
\]
branched along the exceptional divisor
\[
 F:=q(E)\subseteq X_{-},
 \qquad q^{*}F=2E.
\] 
Note that $q$ is finite and flat
of degree two. 
$q_{*}\mathcal O_{\widetilde X}$ is locally free of rank two. Consider the trace morphism associated with the finite morphism $q$. We have an exact sequence 
\[0\to \ker(\operatorname{Tr})\to q_{*}\mathcal O_{\widetilde X}\xrightarrow{\mathrm{Tr}} \mathcal O_{X_{-}}\to 0.\]
The trace morphism splits the inclusion of constants because $\operatorname{Tr}=\operatorname{id}+\iota^*$. Then let \[\mathcal{L}:=\ker(\operatorname{Tr})^{-1}\] be a line bundle. The trace decomposition is
\[
q_{*}\mathcal O_{\widetilde X}
   =\mathcal O_{X_{-}}\oplus\mathcal L^{-1},
\qquad
\mathcal L^{-1}=\ker(\operatorname{Tr}).
\]
The product of two anti-invariant sections is invariant and therefore induces a homomorphism
\[
\mathcal L^{-1}\otimes\mathcal L^{-1}
   \longrightarrow\mathcal O_{X_{-}},
\]
or equivalently a section
\[
\sigma\in H^{0}(X_{-},\mathcal L^{\otimes2}).
\]
On an open set $U\subset X_{-}$ with a frame $e$ of $\mathcal L$,
write $\sigma=f e^{\otimes2}$, and let $t$ denote the element of the
trace-zero summand corresponding to $e^{-1}$. Then
\[
q_{*}\mathcal O_{\widetilde X}|_{U}   \cong\mathcal O_{U}[t]/(t^{2}-f).
\]
Let $F=\operatorname{div}(\sigma)$ be the branch divisor and let
$E\subset\widetilde X$ be the ramification divisor. The section
$\sigma$ gives
\[
\mathcal L^{\otimes2}\cong\mathcal O_{X_{-}}(F).
\]
Moreover, the local sections $t\,q^{*}e$ glue to a global section of
$q^{*}\mathcal L$ whose zero divisor is $E$. Hence
\begin{equation}\label{eq:branch-half}
    q^{*}\mathcal L\cong\mathcal O_{\widetilde X}(E),\qquad q^{*}F=2E.
\end{equation}
Setting
\[
\ell:=c_{1}(\mathcal L)\in H^{2}(X_{-} ),
\]
we obtain
\[
[F]:=c_{1}\!\left(\mathcal O_{X_{-}}(F)\right)=2\ell.
\]

\begin{lemma}\label{lem:exceptional-cubic}\label{prop:Xminus-cubic}
We have \[
 E^{3}=2,\qquad \bigl\langle \ell^{3},[X_{-}]\bigr\rangle=1.
\] In particular, the class $\ell$ is a generator of $H^{2}(X_{-})$.
\end{lemma}

\begin{proof}
For the blow-up of a complex threefold along a smooth curve, let
\[
 \pi\colon E=\mathbb{P}(\mathcal{N}_{C/X_{\mathrm{sph}}})\longrightarrow C,
 \qquad
 \eta :=c_{1}\bigl(\mathcal{O}_{\Pp(\mathcal{N}_{C/X_{\mathrm{sph}}})}(-1)\bigr).
\]Thus, on each fibre $\pi^{-1}(c)\cong\mathbb P^{1}$, the class
$\eta$ has degree $-1$.
Then $\mathcal O_{\widetilde X}(E)|_{E}\cong\mathcal O_{\Pp(\mathcal{N}_{C/X_{\mathrm{sph}}})}(-1)$ and $[E]|_{E}=\eta$. The push-forward formulas give
\[
 \pi_{*}(\eta)=-1,
 \qquad
 \pi_{*}(\eta^{2})=-c_{1}(\mathcal{N}_{C/X_{\mathrm{sph}}}).
\]
By \cref{lem:normal-trivial}, we have $\deg c_{1}(\mathcal{N}_{C/X_{\mathrm{sph}}})=-2$. Consequently,
\[
 E^{3}
 =\int_{E}\eta^{2}
 =-\int_{C}c_{1}(\mathcal{N}_{C/X_{\mathrm{sph}}})=2.
\]
Note that the degree of $q$ is two. By \cref{eq:branch-half} we have
\[
 2\bigl\langle \ell^{3},[X_{-}]\bigr\rangle
 =\bigl\langle q^{*}\ell^{3},[\widetilde X]\bigr\rangle
 =E^{3}=2.
\]
Hence $\ell^{3}=1$ and so $\ell$ is a generator.
\end{proof}

Now we compute the second Stiefel-Whitney class and the Pontryagin functional on $X_{-}$.

\begin{proposition}\label{prop:Xminus-c1}
The first Chern class of $X_{-}$ vanishes.  In particular,
\[
 w_{2}(X_{-})=0.
\]
\end{proposition}

\begin{proof}
The canonical-bundle formulas for blow-up along $C$ and for the
double cover $q$ branched in $F\in|2\mathcal{L}|$ give
\[
 \mathrm{Bl}_C^{*}K_{X_{\mathrm{sph}}}\otimes\mathcal O_{\widetilde X}(E)\cong K_{\widetilde X} \cong q^{*}(K_{X_{-}}\otimes \mathcal{L}).
\]
Using \eqref{eq:branch-half}, their first Chern classes
give $q^{*}c_{1}(K_{X_{-}})=\mathrm{Bl}_C^{*}c_{1}(K_{X_{\mathrm{sph}}})=0$.  
Therefore
\[2c_{1}(K_{X_{-}})=q_{*}q^{*}c_{1}(K_{X_{-}})=0.
\]
By \cref{prop:Xminus-homology}, $H^{2}(X_{-})\cong\Z$ is torsion-free,
so $c_{1}(K_{X_{-}})=0$. Hence 
\[w_2(X_-)=w_2(TX_-)=c_1(-K_{X_-}) \mathrm{mod}\ 2=0.\]
\end{proof}

\begin{proposition}\label{prop:Xminus-p1}
The first Pontryagin functional $p_{X_-}(x)=\langle p_{1}(X_-)\cup x,[X_-]\rangle$ of $X_{-}$ satisfies $p_{X_-}(\ell)=4$.
\end{proposition}

\begin{proof}
Since $F\cong_{\mathrm{diff}}\Pp^1\times \Pp^1$, the signature of $F$ is zero. Moreover $[F]=2\ell$ and $\ell^{3}=1$, hence
\[
 \langle [F]^{3},[X_-]\rangle=\langle (2\ell)^3,[X_-]\rangle=8.
\]
For the oriented real two-plane bundle underlying the normal complex line bundle, $p_{1}(\mathcal{N}_{F/X_{-}})=c_{1}(\mathcal{N}_{F/X_{-}})^{2}$.
The self-intersection formula gives
\[
 c_{1}(\mathcal{N}_{F/X_{-}})=i^{*}[F],
\qquad i\colon F\hookrightarrow X_{-},
\]
and therefore
\[
 \left\langle c_{1}(\mathcal{N}_{F/X_{-}})^{2},[F]\right\rangle  =\langle [F]^{3},[X_-]\rangle=8.
\]
The real normal sequence along $F$ splits, so as real vector bundles $TX_{-}|_{F}\cong TF\oplus \mathcal N_{F/X_{-}}$.
The Hirzebruch signature theorem on the four-manifold $F$ therefore
gives
\begin{align*}
 \left\langle p_{1}(X_{-})\cup \ell,[X_{-}]\right\rangle&=\frac12\bigl\langle p_{1}(X_{-})\cup[F],[X_{-}]\bigr\rangle\\
 &=\frac12\bigl\langle p_{1}(TF),[F]\bigr\rangle
   +\frac12\bigl\langle c_{1}(\mathcal{N}_{F/X_{-}})^{2},[F]\bigr\rangle\\
 &=\frac32\operatorname{sign}(F)+\frac12F^{3}=4.
\end{align*}
\end{proof}

\subsection{Characteristic classes of \texorpdfstring{$X_{+}$}{Xplus}}

Consider the flop described in \cref{eq:global-flop-diagram}. Let $R\subseteq \widehat X$ be the exceptional divisor. By Section~\ref{sec:Xplus-topology}
\[
 C_{-}\cong\mathbb P^{1},
 \qquad
 \mathcal N_{C_{-}/X_-}\cong\mathcal O(-1)\oplus\mathcal O(-1),
 \qquad
 R\cong\Pp^{1}\times\Pp^{1}.
\]
We now compute the characteristic classes of \texorpdfstring{$X_{+}$}{Xplus}.

Since $C$ intersects with $D_-$ in $ Q,  Q'$ transversely, the curve $C_-$ satisfies $F\cdot C_{-}=2$ and $\ell\cdot C_{-}=1$. Choose the rulings of $R\cong\Pp^{1}\times\Pp^{1}$ so that
$\operatorname{Bl}_{C_-}|_{R}$ is the projection to the first factor and $\operatorname{Bl}_{C_+}|_{R}$ is the projection to the second.  The line bundle $\mathcal{L}$ has degree one on
$C_{-}$, and therefore
\begin{equation}\label{lem:branch-flop-intersection}
    \operatorname{Bl}_{C_-}^{*}(\mathcal{L})|_{R}\cong\mathcal O_{R}(1,0),
 \qquad
 \mathcal O_{\widehat{X}}(R)|_{R}\cong\mathcal O_{R}(-1,-1).
\end{equation}
For the blow up along $C_+$, we have
\[
 \Pic(\widehat X)
 \cong \operatorname{Bl}_{C_+}^{*}\Pic(X_{+})\oplus
       \Z[\mathcal O_{\widehat X}(R)],
\]
where the second coefficient is detected by degree on an exceptional
fibre. Indeed, restriction defines a degree homomorphism
$\Pic(\widehat{X})\to\mathbb Z$, the bundle $\mathcal O_{\widehat{X}}(R)$ has degree
$-1$, and the kernel consists precisely of pullbacks: a line bundle
of degree zero restricts on
$R=\mathbb P(\mathcal{N}_{C_{+}/X_{+}})$ as the pullback of a line bundle on
$C_{+}$, and hence descends uniquely across the smooth centre.
Therefore there is a unique holomorphic line bundle $\mathcal{L}_{+}$ on
$X_{+}$ such that \begin{equation}\label{eq:flop-divisor-relation}
    \operatorname{Bl}_{C_+}^{*}\mathcal{L}_{+}\cong \operatorname{Bl}_{C_-}^{*}\mathcal{L}\otimes\mathcal O_{\widehat X}(R).
\end{equation}  Put
\[
 \ell_{+}:=c_{1}(\mathcal{L}_{+}).
\]
Let $U=X_-\setminus C_-=X_+\setminus C_+$ be the complement of the two flopping curves. The excision gives $H^{k}(X_{\pm},U )
 \cong H^{k-4}(C_{\pm} )$.
Hence the relative groups vanish for
$k=2,3$, and the long exact sequences of the pairs give isomorphisms
\[
 H^{2}(X_{-})\cong H^{2}(U)
 \cong H^{2}(X_{+}).
\]
Restricting \eqref{eq:flop-divisor-relation} to the common complement
gives $\ell|_{U}=\ell_{+}|_{U}$. Thus these isomorphisms identify $\ell$ with $\ell_{+}$.

\begin{proposition}\label{prop:flop-cubic}
The generator $\ell_{+}\in H^{2}(X_{+})$ satisfies
\[
 \bigl\langle \ell_{+}^{3},[X_{+}]\bigr\rangle=0.
\]
\end{proposition}

\begin{proof}
Choose the two rulings of
$R\cong\Pp^{1}\times\Pp^{1}$ so that $\operatorname{Bl}_{C_{-}}|_{R}$ is projection to the
first factor.  By \cref{lem:branch-flop-intersection}, we have the intersection formulas
\[
 (\operatorname{Bl}_{C_-}^*\ell)^{2}R=0,
 \qquad
 \operatorname{Bl}_{C_-}^*\ell\,R^{2}=-1,
 \qquad
 R^{3}=2.
\]
Using \cref{eq:flop-divisor-relation} and $\ell^{3}=1$, we obtain
\begin{align*}
 \left\langle \ell_{+}^{3},[X_{+}]\right\rangle
 &=\left\langle(\operatorname{Bl}_{C_+}^{*}\ell_{+})^{3},[\widehat{X}]\right\rangle\\
 &=\left\langle(\operatorname{Bl}_{C_-}^*\ell+R)^{3},[\widehat{X}]\right\rangle\\
 &=\ell^{3}+3(\operatorname{Bl}_{C_-}^*\ell)^{2}R+3\operatorname{Bl}_{C_-}^*\ell\,R^{2}+R^{3}\\
 &=1+0-3+2=0.
\end{align*}
\end{proof}

We next compute the Chern classes and Pontryagin classes after the flop.

\begin{proposition}\label{prop:Xplus-spin}
The first Chern class of $X_{+}$ vanishes.  In particular,
\[
 w_{2}(X_{+})=0.
\]
\end{proposition}

\begin{proof}
The canonical-bundle formulas for the two blow-ups give
\[
 K_{\widehat{X}}\cong \operatorname{Bl}_{C_-}^*K_{X_{-}}\otimes\mathcal O_{\widehat{X}}(R)
 \cong \operatorname{Bl}_{C_+}^{*}K_{X_{+}}\otimes\mathcal O_{\widehat{X}}(R).
\]
Thus $\operatorname{Bl}_{C_-}^*c_{1}(K_{X_{-}})=\operatorname{Bl}_{C_+}^{*}c_{1}(K_{X_{+}})$. The left side is zero by \cref{prop:Xminus-c1}, and pullback by a blow-up is
injective in degree two by \cref{lem:integral-blowup}.  Hence
$c_{1}(K_{X_{+}})=0$, so
$c_{1}(TX_{+})=-c_{1}(K_{X_{+}})=0$.  Reduction modulo two gives
$w_{2}(X_{+})=0$.
\end{proof}

\begin{proposition}\label{prop:Xplus-p1}
The first Pontryagin functional $p_{X_+}(x)=\langle p_{1}(X_+)\cup x,[X_+]\rangle$ of $X_{+}$ is identically zero.
\end{proposition}

\begin{proof}
We first compare the Euler characteristics of $\mathcal{L}$ and $\mathcal{L}_{+}$.
Restricting \cref{eq:flop-divisor-relation} to $R$, we have
\[
 \operatorname{Bl}_{C_+}^*\mathcal{L}_+|_R\cong \mathcal O_{R}(1,0)\otimes\mathcal O_{R}(-1,-1)
 \cong\mathcal O_{R}(0,-1),
\]
whose cohomology vanishes in every degree. The divisor exact sequence
\[
 0\longrightarrow \operatorname{Bl}_{C_-}^*\mathcal{L}
 \longrightarrow \operatorname{Bl}_{C_-}^*\mathcal{L}\otimes\mathcal O_{\widehat{X}}(R)
 \longrightarrow\mathcal O_{R}(0,-1)
 \longrightarrow0
\]
therefore gives
\[
 \chi(\widehat{X},\operatorname{Bl}_{C_+}^{*}\mathcal{L}_{+})=\chi(\widehat{X},\operatorname{Bl}_{C_-}^*\mathcal{L}).
\]
For a blow-up along a smooth centre,
\[
 R^{i}\operatorname{Bl}_{C_+,*}\mathcal O_{\widehat{X}}=R^{i}\operatorname{Bl}_{C_-,*}\mathcal O_{\widehat{X}}=0
 \quad(i>0),
\]
and both direct images in degree zero are the structure sheaves. The projection formula consequently yields
\begin{equation*}
 \chi(X_{+},\mathcal{L}_{+})=\chi(X_{-},\mathcal{L}).
\end{equation*}

For the compact complex threefold $X_+$, since $c_{1}(TX_+)=0$, the
Hirzebruch--Riemann--Roch formula gives
\begin{align*}
 \chi(X_+,\mathcal L_+)
 &=\frac{1}{6}\langle l_+^{3},[X_+]\rangle
   +\frac{1}{12}\langle c_{2}(X_+)\cup l_+,[X_+]\rangle\\
 &=\frac{1}{24}
   \bigl\langle 4l_+^{3}-p_{1}(X_+)\cup l_+,[X_+]\bigr\rangle,
\end{align*}
because $p_{1}=c_{1}^{2}-2c_{2}=-2c_{2}$.  The same formula applies to $\chi(X_{-},\mathcal{L})$, hence
\[
 24\chi(X_{-},\mathcal{L})=4\left\langle \ell^{3},[X_{-}]\right\rangle
 -
 \left\langle p_{1}(X_{-})\cup \ell,[X_{-}]\right\rangle
 =4-4=0
\]
by \cref{prop:Xminus-cubic,prop:Xminus-p1}. On $X_{+}$,
$\ell _{+}^{3}=0$ by \cref{prop:flop-cubic}. Hence
\[
 \bigl\langle p_{1}(X_{+})\cup \ell_{+},[X_{+}]\bigr\rangle=0.
\]
\end{proof}

\section{Proof of the main theorems}\label{Proof of the main theorems}

We use Wall's smooth classification theorem for spin six-manifolds
\cite{Wall}. Jupp established the corresponding broader classification of simply connected six-manifolds with torsion-free integral homology, including the non-spin case \cite[Theorem~1 and the corollary on p.~299]{Jupp}.

We now state precisely the form of Wall's result used in this paper.

\begin{theorem}\cite[Theorem~5]{Wall}
\label{thm:wall-specialized}
An admissible spin Wall system in dimension six is a quadruple
\[
       (r,H,\mu,p)
\]
consisting of a nonnegative integer $r$, a finitely generated free
abelian group $H$, a symmetric trilinear form
\[
       \mu\colon H\times H\times H\longrightarrow\mathbb Z,
\]
and a homomorphism $p\colon H\to\mathbb Z$, subject to the
congruences
\begin{align}
 \mu(x,x,y)+\mu(x,y,y)&\equiv0\pmod 2
       &&\text{for all }x,y\in H,\label{eq:wall-parity}\\
 4\mu(x,x,x)-p(x)&\equiv0\pmod{24}
       &&\text{for all }x\in H.\label{eq:wall-index}
\end{align}

For a closed, connected, oriented, simply connected spin smooth real six-manifold $Y$ with torsion-free integral homology, set
\[
 H_{Y}:=H^{2}(Y ),\qquad
 r_{Y}:=\frac{1}{2}\operatorname{rank}H^{3}(Y ),
\]
\[
 \mu_{Y}(x,y,z):=\langle x\cup y\cup z,[Y]\rangle,
 \qquad
 p_{Y}(x):=\langle p_{1}(Y)\cup x,[Y]\rangle .
\]
Then $r_{Y}$ is an integer and
$(r_{Y},H_{Y},\mu_{Y},p_{Y})$ is admissible. The assignment
\[
 Y\longmapsto(r_{Y},H_{Y},\mu_{Y},p_{Y})
\]
induces a bijection from orientation-preserving diffeomorphism classes
of such manifolds to isomorphism classes of admissible systems.
Explicitly, two such manifolds $Y$ and $Y'$ are
orientation-preservingly diffeomorphic if and only if
$r_{Y}=r_{Y'}$ and there is an isomorphism
$\varphi\colon H_{Y}\to H_{Y'}$ satisfying
\[
 \mu_{Y'}(\varphi x,\varphi y,\varphi z)=\mu_{Y}(x,y,z),
 \qquad
 p_{Y'}(\varphi x)=p_{Y}(x)
\]
for all $x,y,z\in H_{Y}$. Conversely, every admissible system is
realized by such a manifold.
\end{theorem}

\subsection{The proof of \cref{thm:Xminus-main}}

\begin{proof}[Proof of \cref{thm:Xminus-main}]
By \cref{prop:Xminus-homology}, the manifold $X_{-}$ is simply
connected and has torsion-free integral homology, with
\[
 H^{2}(X_{-} )=\mathbb Z \ell,
 \qquad
 H^{3}(X_{-} )=0.
\]
Thus $r_{X_{-}}=0$.  The calculations in
\cref{prop:Xminus-cubic,prop:Xminus-c1,prop:Xminus-p1} give
\[
 \mu_{X_{-}}(\ell,\ell,\ell)=1,
 \qquad
 w_{2}(X_{-})=0,
 \qquad
 p_{X_{-}}(\ell)=4.
\]

Let $u=c_{1}(\mathcal O_{\mathbb P^{3}}(1))$ be the first Chern class of the standard complex structure on $\mathbb{P}^3$. Then
\[ u^3=1.\qquad
 w_{2}(\mathbb P^{3})=0,
 \qquad
 p_{1}(\mathbb P^{3})
 =c_{1}^{2}-2c_{2}=4u^{2}.
\]
Hence $X_{-}$ and $\mathbb P^{3}$ have isomorphic Wall systems,
the isomorphism being $\ell\mapsto u$.  By
\cref{thm:wall-specialized},
\[
 X_{-}\cong_{\mathrm{diff}}\mathbb P^{3}.
\]

This complex structure is not the standard one.  Indeed,
\cref{prop:Xminus-c1} gives $c_{1}(TX_{-})=0$, whereas the standard
complex structure on $\mathbb P^{3}$ has first Chern class $4u$.

We next compare it with the construction of Huckleberry, Kebekus and Peternell. Given a complex structure $X$ on $S^{6}$ and a point $p\in X$, they consider
\[
 \pi_{p}\colon Z_{p}:=\Bl_{p}X\longrightarrow X
\]
and observe that $Z_{p}$ is diffeomorphic to $\mathbb P^{3}$;
their automorphism theorem then yields a one-dimensional family of such
complex structures \cite[p.~102 and Corollary~1.2]{HKP}.  Let
$E_{p}\cong \mathbb P^{2}$ be the exceptional divisor and put
$e_{p}=[E_{p}]$.  The point-blow-up formula gives
\[
 H^{2}(Z_{p} )=\mathbb Z e_{p},
 \qquad
 e_{p}^{3}=1,
\]
and
\[
 K_{Z_{p}}
 \cong \pi_{p}^{*}K_{X}\otimes\mathcal O_{Z_{p}}(2E_{p}).
\]
Since $H^{2}(X )=0$, it follows that
\[
 c_{1}(TZ_{p})=-2e_{p}\ne0.
\]
Consequently $X_{-}$ is not biholomorphic to any of these point
blow-ups.  Nor are they deformation equivalent through smooth compact
complex manifolds since in a smooth proper holomorphic family, the first Chern class of the relative tangent bundle forms a locally constant class.  A zero first Chern class therefore cannot
deform to $-2e_{p}$ or $4u$.  This proves all the assertions of
\cref{thm:Xminus-main}.
\end{proof}

\begin{proposition}\label{prop: Q3exotic}
The complex threefold $\widetilde X$ is orientation-preservingly diffeomorphic, but not biholomorphic, to the smooth quadric threefold $Q_{3}\subset\mathbb P^{4}$.
\end{proposition}

\begin{proof}
Computing the topology invariants of $\widetilde X$, we have $E^{3}=2, w_{2}(\widetilde X)=[E]\bmod2, \bigl\langle p_{1}(\widetilde X)[E],[\widetilde X]\bigr\rangle=2$. By Wall--Jupp's classification, and comparing to the topology invariants of $Q_3$, we find that $\widetilde X\cong_{\mathrm{diff}} Q_3$.

The first Chern
class of $\widetilde X$ is
\[
c_{1}(T\widetilde X)=-[E],
\]
which is primitive in $H^{2}(\widetilde X )$.  By contrast,
\[
c_{1}(TQ_{3})=3c_{1}\bigl(\mathcal O_{Q_{3}}(1)\bigr)
\]
has divisibility three.  A biholomorphism preserves the first Chern
class and its divisibility in integral cohomology.  Hence
$\widetilde X$ cannot be biholomorphic to $Q_{3}$.
\end{proof}

\begin{remark}\label{rem:point-blowup-flop}

The point-blow-up construction does not yield a second route to $S^{2}\times S^{4}$. Indeed, for $p\notin\{ Q, Q'\}$, let
\[
Z_{p}:=\operatorname{Bl}_{p}(X_{\mathrm{sph}}).
\]
We may choose a double curve $D\subset X_{\mathrm{sph}}$ disjoint from $p$ and perform the Atiyah flop along its strict transform in $Z_{p}$. The resulting complex threefold, however, is not diffeomorphic to $S^{2}\times S^{4}$; its underlying smooth manifold is again $\mathbb{P}^{3}$ (again can be proved by Wall's classification), endowed with an exotic complex structure.

\end{remark}

\subsection{The proof of \cref{thm:Xplus-main}}

\begin{proof}[Proof of \cref{thm:Xplus-main}]
By \cref{thm:homology}, $X_{+}$ is simply connected and its integral
homology is torsion-free, with
\[
 H^{2}(X_{+})=\Z \ell_{+},
 \qquad
 H^{3}(X_{+})=0.
\]
Hence $r_{X_{+}}=0$. By
\cref{prop:flop-cubic,prop:Xplus-spin,prop:Xplus-p1},
\[
 \mu_{X_{+}}(\ell_{+},\ell_{+},\ell_{+})=0,
 \qquad
 w_{2}(X_{+})=0,
 \qquad
 p_{X_{+}}(\ell_{+})=0.
\]

Let $x=\operatorname{pr}_{1}^{*}a$ generate
$H^{2}(S^{2}\times S^{4})$, and let
$y=\operatorname{pr}_{2}^{*}b$ generate its fourth cohomology, with
\[
 \left\langle x\cup y,[S^{2}\times S^{4}]\right\rangle=1.
\]
The Künneth theorem gives
\[
 x^{2}=0,\qquad H^{3}(S^{2}\times S^{4} )=0.
\]
Moreover every sphere is stably parallelizable:
\[
 TS^{n}\oplus\mathbb R\cong\mathbb R^{n+1}.
\]
Consequently
\[
 T(S^{2}\times S^{4})\oplus\mathbb R^{2}
 \cong
 \operatorname{pr}_{1}^{*}(TS^{2}\oplus\mathbb R)
 \oplus
 \operatorname{pr}_{2}^{*}(TS^{4}\oplus\mathbb R)
\]
is trivial. Stable characteristic classes therefore give
\[
 w_{2}(S^{2}\times S^{4})=0,
 \qquad
 p_{1}(S^{2}\times S^{4})=0.
\]
Thus $S^{2}\times S^{4}$ has Wall system
\[
       (0,\mathbb Z,0,0),
\]
and the isomorphism $\ell_{+}\mapsto x$ identifies it with the Wall
system of $X_{+}$. By \cref{thm:wall-specialized}, there is an
orientation-preserving diffeomorphism
\[
       X_{+}\cong_{\mathrm{diff}} S^{2}\times S^{4}.
\]
Since $X_{+}$ is a complex threefold by the flop construction in
\cref{def:Xplus-flop}, this proves \cref{thm:Xplus-main}.
\end{proof}

\section{Further discussion}\label{Further discussion}

\subsection{Dependence on the flopping curve}

Retain the labelling of \cite[Proposition~4.6(iii)]{S6Paper}, so that
the three cusp double curves correspond to the unoriented directions
$e_{1},e_{2},e_{1}-e_{2}$, and the fixed curve is the one of direction
$e_{2}$ by \cite[Proposition~9.24]{S6Paper}.  After blowing up the fixed
curve and taking the $\mu _2$-quotient, the strict transforms of the
other two double curves give disjoint curves
\[
 C_{-},C_{-}'\subseteq X_{-},\] with normal bundles $\mathcal O_{\mathbb P^{1}}(-1)^{\oplus2}$.
Thus either curve could be flopped. The same Wall-invariant calculation shows that both resulting complex manifolds are diffeomorphic to $S^2\times S^4$.

This local toric symmetry does not by itself extend to the completed
global family.  The connected automorphism group computed in
\cite[Proposition~9.23]{S6Paper} is the vertical $\mathbb C^{*}$, and
its toric action preserves each orbit closure, hence each double curve
separately; see the proof of \cite[Proposition~9.24]{S6Paper}.
We therefore do not yet know whether the two resulting complex manifolds are
biholomorphic, or even deformation equivalent. We leave this as a
natural question for further study.

\subsection{The role of the source construction}

The construction of $X_{\mathrm{sph}}$ in \cite{S6Paper} is extensive:
it develops the period family, the toric cusp filling, the two logarithmic
transforms, the integral topology, and a detailed analysis of the
automorphism group.  Reading and checking the entire construction is
correspondingly demanding.  Once its existence and diffeomorphism
theorems are taken as input, not every one of those calculations is
logically required for a direct proof of \cref{thm:Xplus-main}.

The additional material is nevertheless geometrically fruitful.  In
particular, the vertical $\mathbb C^{*}$-action, its fixed curve, the
three double curves of the cusp fibre, and the explicit $A_{2}$-fan are
precisely the data that reveal the quotient resolution and the floppable
curve used here.  These features are invisible in the bare statement
that $S^{6}$ carries a complex structure.  Thus the breadth of the
source construction should not be regarded merely as redundancy: its
more exploratory calculations expose birational geometry that suggests
new constructions.  It would be interesting to determine whether other
parts of that geometry lead to further compact complex manifolds or to
new complex structures on familiar smooth six-manifolds.

Recently, the main result of \cite{S6Paper} has been claimed to be formalized by Lean 4 by Boris Alexeev. See \href{https://github.com/plby/HopfProblem}{https://github.com/plby/HopfProblem}.

\begingroup
\bibliographystyle{abbrv}
\bibliography{references}
\endgroup

\end{document}